\documentclass[11pt]{amsart}
\usepackage{float}
\usepackage[usenames]{color}
\usepackage{graphicx}
\usepackage{amssymb}
\theoremstyle{plain}
\newtheorem{thm}{Theorem}[section]
\newtheorem{lemma}[thm]{Lemma}

\newtheorem{conj}[thm]{Conjecture}
\newtheorem{prop}[thm]{Proposition}

\theoremstyle{definition}

\newtheorem{rmk}[thm]{Remark}
\newtheorem{example}[thm]{Example}

\newcommand{\frg}{\mathfrak{g}}
\newcommand{\frh}{\mathfrak{h}}

\newcommand{\bbC}{\mathbb{C}}

\newcommand{\bbN}{\mathbb{N}}

\newcommand{\bbZ}{\mathbb{Z}}

\begin{document}

\title[Orbits of antichains]
{Orbits of antichains in certain root posets}

\author{Chao-Ping Dong}

\address[Dong]
{Institute of Mathematics,  Hunan University,
Changsha 410082, China}
\email{chaoping@hnu.edu.cn}

\author{Suijie Wang}
\address[Wang]
{Institute of Mathematics,  Hunan University,
Changsha 410082, China}
\email{wangsuijie@hnu.edu.cn}
\thanks{The research is supported by NSFC grant 11571097, 11401196, and  the China Scholarship Council.}

\abstract{This paper gives another proof of Propp and Roby's theorem saying that the average antichain size in any reverse operator orbit of the poset $[m]\times [n]$ is $\frac{mn}{m+n}$. It is conceivable that our method should work for other situations. As a demonstration, we show that the average size of antichains in any reverse operator orbit of $[m]\times K_{n-1}$  equals $\frac{2mn}{m+2n-1}$. Here $K_{n-1}$ is the minuscule poset $[n-1]\oplus ([1] \sqcup [1]) \oplus [n-1]$. Note that $[m]\times [n]$ and $[m]\times K_{n-1}$ can be  interpreted as sub-families of certain root posets. We guess these root posets should provide a unified setting to exhibit the homomesy phenomenon defined by Propp and Roby.}
\endabstract



\maketitle

%\tableofcontents

\section{Introduction}

Recently, the study of orbit structure of the reverse operator, which is also called rowmotion \cite{SW} or Fon-Der-Flaass action \cite{RS} in the literature, has emerged at the front lines of inquiry.
This paper introduces a method that aims to calculate the sizes of antichains in each orbit. Our method turns out to be effective in dealing with certain root posets coming from the classical Lie algebras.

Now let us be more precise. A subset $I$ of a finite poset $(P, \leq)$ is called an \emph{ideal} if $x\leq y$ in $P$ and $y\in I$  implies that $x\in I$.  Let $J(P)$ be the set of ideals of $P$, partially ordered by inclusion. A subset $A$ of $P$ is an \emph{antichain} if its elements are mutually incomparable. We
collect the antichains of $P$ as $\mathrm{An}(P)$. For any $x\in P$,
let $I_{\leq x}=\{y\in P\mid y\leq x\}$. Given an ideal $I$ of $P$, set $\Gamma(I):=\max(I)$. For an antichain $A$ of
$P$, let $I(A)=\bigcup_{a\in A} I_{\leq a}$. The \emph{reverse
operator} $\Psi$ is defined by $\Psi(A)=\min
(P\setminus I(A))$. This operator was introduced on hypergraphs by Duchet in 1974 \cite{D}, and studied on $[m]\times [n]$---the product of two chains---by Fon-Der-Flaass in 1993 \cite{F}.
Since antichains of $P$ are in bijection with
 ideals of $P$, the reverse operator acts on
 ideals of $P$ as well. For convenience, we may just call an $\Psi_{P}$-orbit of antichains or ideals  an \emph{orbit}.

Recently, the following result has been established in \cite{PR}.

\medskip
\noindent\textbf{Theorem A.} \textbf{(Propp and Roby)}
\emph{The average size of antichains in any $\Psi$-orbit
of $[m]\times [n]$ equals $\frac{mn}{m+n}$.}
\medskip

Later,  Theorem A has been extended uniformly to all minuscule posets by Rush and Wang \cite{RW}. Recall that as been classified by Proctor \cite{Pr}, the minuscule posets include two exceptions $J^2([2]\times [3])$, $J^3([2]\times [3])$, and the following three infinite families:
\begin{itemize}
\item[$\bullet$] $[m]\times [n]$;
\item[$\bullet$] $H_n:=[n]\times [n]/S_2$;
\item[$\bullet$] $K_r:=[r]\oplus ([1] \sqcup [1]) \oplus [r]$ (the ordinal sum, see page 246 of Stanley \cite{St}).
\end{itemize}

The first aim of this paper is to give another proof of Theorem A. It is conceivable that the method developed in Sections 3 and 4 can be applied to more general situations. As a demonstration, we report the following.

\medskip
\noindent\textbf{Theorem B.}
\emph{The average size of antichains in any $\Psi$-orbit of $[m]\times K_{n-1}$  equals $\frac{2mn}{m+2n-1}$.}
\medskip

Theorem B follows from Propositions \ref{prop-D-I} and \ref{prop-D-II}. In Section 7, we will recall certain root posets $\Delta(1)$. They consists of the infinite families $[m]\times [n]$, $H_n$, $[m]\times K_{n-1}$, and twenty other posets.  By Theorems A and B, and by checking the twenty posets directly, one knows that the average size of antichains in any $\Psi$-orbit of $\Delta(1)$  equals $\frac{\#\Delta(1)}{d+1}$, where $d$ is the maximum height of the roots in $\Delta(1)$.

The paper is outlined as follows: Section 2 prepares necessary preliminaries. Section 3 aims to prove Theorem A via another method, where the proof of a key lemma is postponed to Section 4. Sections 5 and 6 are devoted to proving Theorem B. More precisely, they handle orbits in $[m]\times K_{n-1}$ of type I and II, respectively. Section 7 recalls the root posets $\Delta(1)$, and raises two conjectures pertaining to their orbit structure. Our conjecures simply aim to guess that $\Delta(1)$ would provide a unified setting to exhibit the homomesy phenomenon invented by Propp and Roby.

\section{Preliminary results}
Throughout this paper, we let $\mathbb{N}=\{0, 1, 2, \dots\}$, and
$\mathbb{P}=\{1, 2, \dots\}$. For each $k\in\mathbb{P}$, the poset
$[k]:=\{1,2, \dots, k\}$ is equipped with the order-reversing involution $c$ such that
$c(i)=k+1-i$. For  $s, t\in\mathbb{P}$ such that $s<t$, $[s, t]:=\{s, s+1, \dots, t\}$.

As on page 244 of Stanley \cite{St}, a finite poset $P$ is said to be
\emph{graded} if every maximal chain in $P$ has the same length. In
this case, there is a unique rank function $r$ from $P$ to the
positive integers $\mathbb{P}$ such that all the minimal elements
have rank $1$, and $r(x)=r(y)+1$ if $x$ covers $y$.
We always assume that the finite poset $P$ is graded.
Let $d$ be the maximum value of $r$ on $P$. For $1\leq j \leq d$, let $P_j$ be the $j$-th rank level of $P$. That is, $P_j=\{p\in P \mid r(p)=j\}$. Set $P_0$ to be the empty set. Note that each $P_j$ is an antichain. Put $L_i=\bigsqcup_{j=1}^{i} P_j$ for $1\leq i\leq d$, and let $L_0$
be the empty set. We call those $L_i$ \emph{rank ideals}. Recall that the reverse operator
acts on ideals as well. For instance,
$\Psi_P(L_i)=L_{i+1}$, $0\leq i<d$ and $\Psi_P(L_d)=L_{0}$.

It is well known that a general
ideal of $[m]\times P$ can be identified with $(I_1, \dots, I_m)$, where each
$I_i$ is an ideal of $P$ and $I_m\subseteq  \cdots \subseteq I_{1}$. We say
that the ideal $(I_1, \dots, I_m)$ is \emph{full rank} if
each $I_i$ is a rank ideal in $P$.  Let $\mathcal{O}(I_1, \dots, I_m)$ be
the $\Psi_{[m]\times P}$-orbit of $(I_1, \dots, I_m)$. The following lemma is taken from Section 2 of \cite{DW}. One may also verify it directly from the definition of $\Psi$.

\begin{lemma}\label{lemma-operator-ideals-CmP}
Keep the notation as above.  Then
for any $n_0\in \bbN$, $n_i\in\mathbb{P}$ ($1\leq i\leq s$) such that $\sum_{i=0}^{s} n_i =m$, we have
\begin{equation}\label{rank-level}
\Psi_{[m]\times P}(L_d^{n_0}, L_{i_1}^{n_1}, \dots, L_{i_s}^{n_s})=
(L_{i_1+1}^{n_0+1}, L_{i_2+1}^{n_1}, \dots, L_{i_s+1}^{n_{s-1}}, L_0^{n_s-1}),
\end{equation}
where $0\leq i_s<\cdots <i_1<d$, $L_d^{n_0}$ denotes $n_0$ copies of $L_d$ and so on.
\end{lemma}

In view of Lemma \ref{lemma-operator-ideals-CmP}, the ideal $(I_1, \dots, I_m)$ is full rank if and only if each ideal in its $\Psi_{[m]\times P}$-orbit is full rank.
Therefore, it makes sense to say that the orbit $\mathcal{O}(I_1, \dots, I_m)$
is of \emph{type I} if $(I_1, \cdots, I_m)$ is full rank, and to say it
is of \emph{type II} otherwise.

\begin{figure}[H]
\centering \scalebox{0.24}{\includegraphics{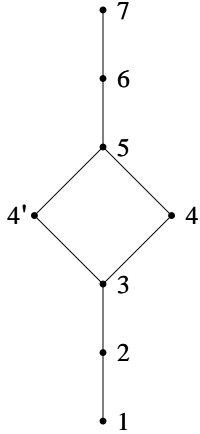}}
\caption{The labeled Hasse diagram of $K_3$}
\end{figure}

In this paper, we care most about the case that $P=K_{n-1}$, whose elements are labelled by $1$, $2$, $\dots$, $n-1$, $n$,
$n^{\prime}$, $n+1$, $\dots$, $2n-2$, $2n-1$. Fig.~1 illustrates
the labeling for $K_3$. Note that $L_i$ ($0\leq i\leq 2n-1$) are all
the full rank ideals. For instance, we have $L_{n}=\{1, 2,
\dots, n, n^{\prime}\}$. Moreover, we put $I_{n}=\{1,  \dots, n-1,
n\}$ and $I_{n^{\prime}}=\{1,  \dots, n-1, n^{\prime}\}$. The following lemma is taken from Section 2 of \cite{DW}. Again, it can be checked directly.

\begin{lemma}\label{lemma-operator-ideals-CmK}
Fix $n_0\in \bbN$, $n_i\in\mathbb{P}$ ($1\leq i\leq s$),
$m_j\in\mathbb{P}$ ($0\leq j\leq t$) such that $\sum_{i=0}^{s} n_i +
\sum_{j=0}^{t} m_j=m$. Take any $0\leq j_t< \cdots<j_1<n\leq
i_s<\cdots <i_1<2n-1$, we have
\begin{align*}
\Psi_{[m]\times K_{n-1}}&(L_{2n-1}^{n_0}, L_{i_1}^{n_1}, \dots, L_{i_s}^{n_s}, I_n^{m_0}, L_{j_1}^{m_1}, \dots, L_{j_t}^{m_t})=\\
&\begin{cases}
( L_{i_1+1}^{n_0+1}, L_{i_2+1}^{n_1}, \dots, L_{i_s+1}^{n_{s-1}}, I_{n^{\prime}}^{n_s}, L_{j_1+1}^{m_0},
L_{j_2+1}^{m_1}, \dots, L_{j_t+1}^{m_{t-1}}, L_0^{m_t-1} ) & \mbox { if } j_1 < n-1;\\
( L_{i_1+1}^{n_0+1}, L_{i_2+1}^{n_1}, \dots, L_{i_s+1}^{n_{s-1}}, L_{n}^{n_s}, I_n^{m_0},
\, \, \, \, L_{j_2+1}^{m_1}, \dots, L_{j_t+1}^{m_{t-1}}, L_0^{m_t-1} )& \mbox { if } j_1 = n-1.
\end{cases}
\end{align*}
\end{lemma}

Keeping the notation of Lemma \ref{lemma-operator-ideals-CmK}, one sees that any non-full-rank ideal of $[m]\times K_{n-1}$ either has the form
\begin{equation}\label{non-full-rank-1}
(L_{2n-1}^{n_0}, L_{i_1}^{n_1}, \dots, L_{i_s}^{n_s}, I_n^{m_0}, L_{j_1}^{m_1}, \dots, L_{j_t}^{m_t})
\end{equation}
or the form
\begin{equation}\label{non-full-rank-2}
(L_{2n-1}^{n_0}, L_{i_1}^{n_1}, \dots, L_{i_s}^{n_s}, I_{n^\prime}^{m_0}, L_{j_1}^{m_1}, \dots, L_{j_t}^{m_t}).
\end{equation}
We say that the above two ideals are \emph{dual} to each other and use ``$\sim$" to denote this relation. It is immediate from Lemma \ref{lemma-operator-ideals-CmK} that  $I\sim I^\prime$ implies that $\Psi(I)\sim \Psi(I^\prime)$, and thus $\Psi^i(I_1)\sim \Psi^i(I_2)$ for any $i\in \bbZ$. In such a case, we say that the two orbits are dual to each other, and write this relation as $\mathcal{O}(I)\sim \mathcal{O}(I^\prime)$.

\section{Another proof of Propp and Roby's theorem}

This section aims to give another proof of Theorem A. We use an NE-SW path to separate an ideal $I$ from its complement. Here NE stands for northeast, and  SW  stands for southwest.
Such a path becomes a binary word when we interpret an NE step by 1 and a SW  step by 0. This process associates a binary word $\Theta(I)$ to each $I\in J([m]\times [n])$.

\begin{example}
 We present the NE-SW  paths  in blue lines for three lower ideals of $[2]\times [3]$ in Fig.~2. Their associated binary words are 01101, 10110 and 11001, respectively. \hfill\qed
\end{example}

\begin{figure}[H]
\centering \scalebox{0.4}{\includegraphics{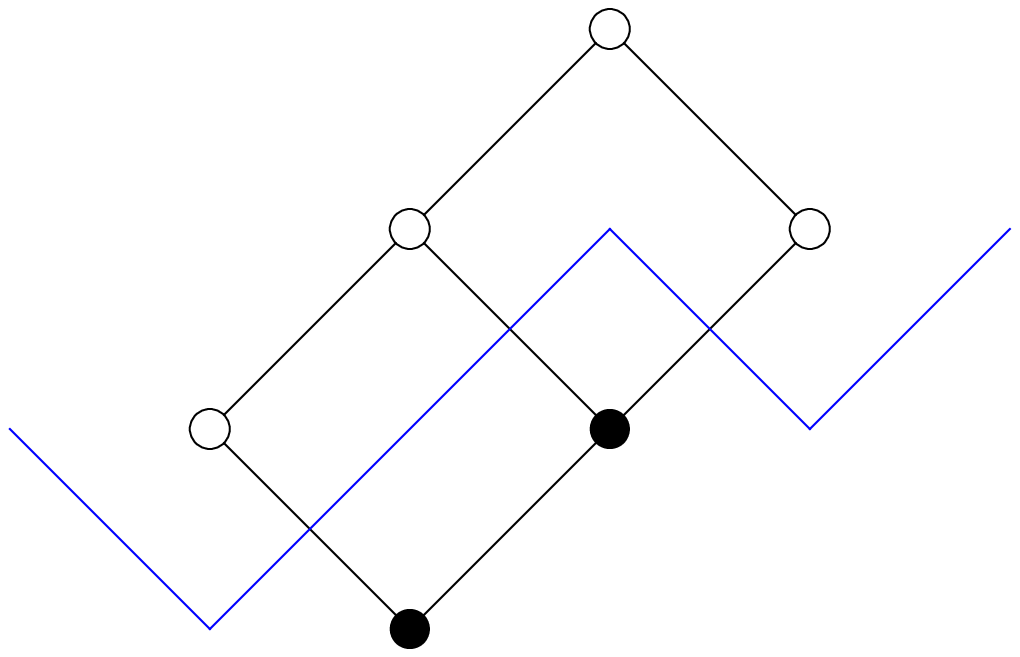}}
\quad\scalebox{0.4}{\includegraphics{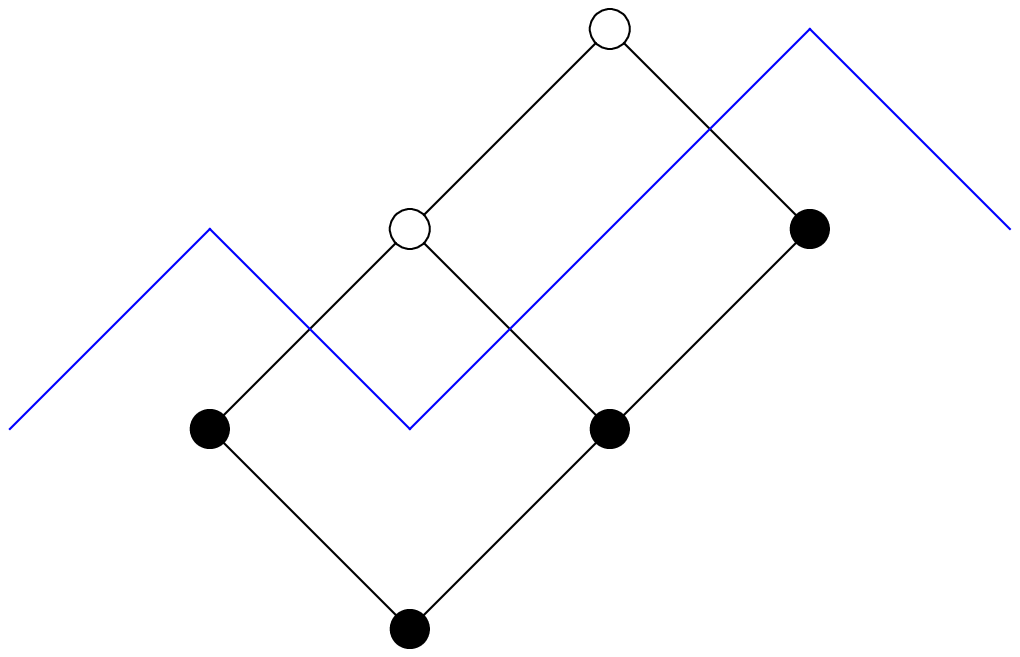}}
\quad
\scalebox{0.4}{\includegraphics{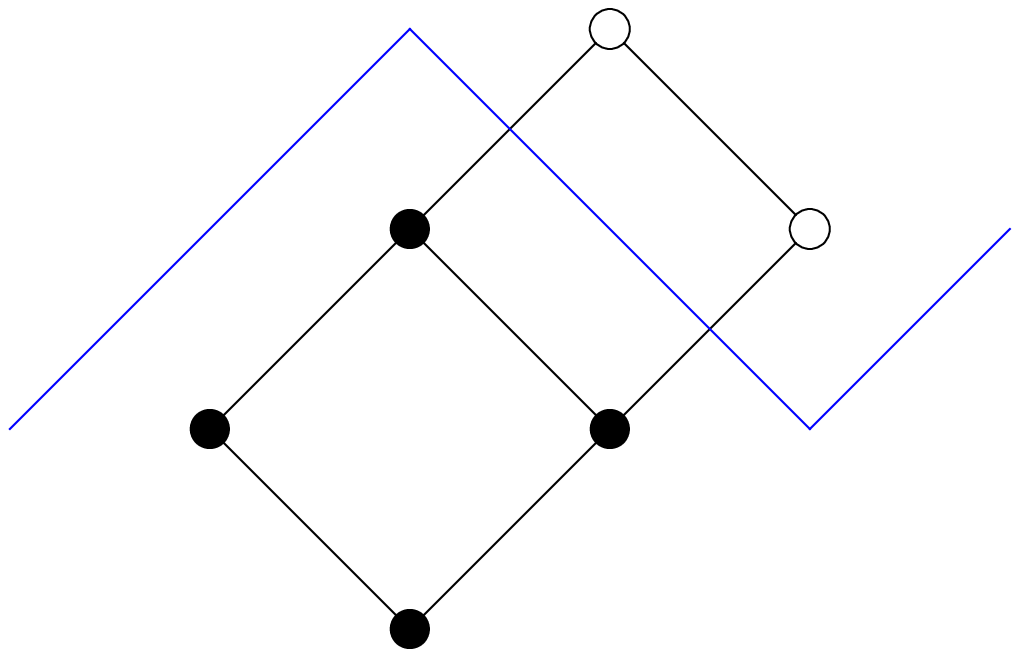}}
\caption{NE-SW paths for three ideals of $[2]\times [3]$}
\end{figure}

Let $B(m, n)$ be the set of binary words with $m$ 0's and $n$ 1's. One sees easily that $\Theta: J([m]\times [n])\to B(m, n)$ is a bijection.

For any integer $a$, we put $1^{a}:=\underbrace{1\dots 1}_{a}$, which is interpreted as the empty word if $a\leq 0$. The notation $0^{a}$ is defined similarly.
A general binary word $w$ in $B(m, n)$ has the form
\begin{equation}\label{binary-word}
w=1^{a_1}0^{b_1}\dots
 1^{a_i} 0^{b_i}\dots 1^{a_s} 0^{b_s},
\end{equation}
where $s\geq 2$, each $a_i$, $b_i$ is positive except for that  $a_1$ and $b_s$ may be $0$.
We define
\begin{equation}\label{psi}
\psi: B(m, n) \to B(m, n)
\end{equation}
via
$$
\psi(w)=0^{b_1-1}1^{a_1+1}\dots
 0^{b_i} 1^{a_i}\dots 0^{b_s+1}1^{a_s-1}.
$$

The following lemma is essentially a translation of Lemma \ref{lemma-operator-ideals-CmP} into the language of binary words. For convenience of reader, we provide a proof.

\begin{lemma}\label{lemma-Pan-binary-word}
Let $I$ be an ideal of $[m]\times [n]$, we have
\begin{equation}
\Theta(\Psi(I))=\psi(\Theta(I)).
\end{equation}
\end{lemma}
\begin{proof}
Suppose that $\Theta(I)=w$ has the form \eqref{binary-word}.
Without loss of generality, we assume that $s\geq 3$ and focus on a middle part $1^{a_i} 0^{b_i}$, where $1<i<s$. We draw the corresponding portion of the path for $I$ in red line in Fig.~3, where dots stand for elements in $I$, while circles stand for elements in its complement. Note that the NE (resp. SW) thick red line segment has length $a_i$ (resp. $b_i$). Then two elements of the antichain $\Gamma(\Psi(I))$ are shown in blue circles, and the corresponding portion of the path for $\Psi(I)$ is drawn in blue. Now the SW (resp. NE) thick blue line segment has length $b_i$ (resp. $a_i$), as desired. We omit the similar analysis for the first part $1^{a_1} 0^{b_1}$ and the last part $1^{a_s}0^{b_s}$ of $\Theta(I)$.
\end{proof}

\begin{rmk}
Note that in \eqref{binary-word}, we have $\sum_{i=1}^{s} a_i=n$ and $\sum_{i=1}^{s} b_i=m$.
Moreover, $s=1$ in \eqref{binary-word} if and only if $\Theta(I)=1^n 0^m$, if and only if $I=[m]\times [n]$. In this case, $\Psi(I)$ is the empty ideal with the associated binary word $0^m1^n$.
\end{rmk}

\begin{figure}[H]
\emph{}\centering \scalebox{0.6}{\includegraphics{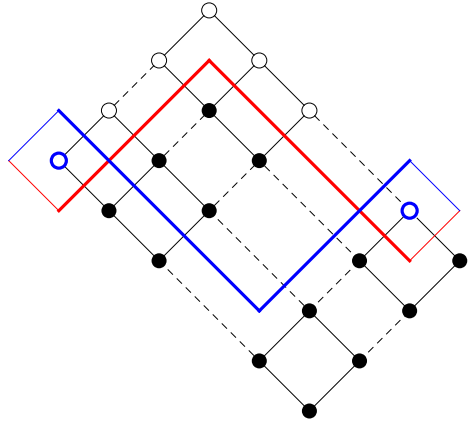}} \caption{Part of the NE-SW paths}
\end{figure}

Let the binary word of $I$ be given by \eqref{binary-word}. Assume that $a_1\geq 1$
and $b_s\geq 1$. The pattern of 0's and 1's are
\begin{equation}\label{initial-zero-pattern}
-0^{b_1}-\cdots -0^{b_i}- \cdots - 0^{b_s}
\end{equation}
and
\begin{equation}\label{initial-one-pattern}
1^{a_1}- \cdots -1^{a_i}- \cdots -1^{a_s}-,
\end{equation}
respectively. There is a unique way to combine \eqref{initial-zero-pattern} and \eqref{initial-one-pattern} in the zig-zag fashion, after which we recover $\Theta(I)$.

If $b_1\geq 2$ and $a_s\geq 2$, then by Lemma \ref{lemma-Pan-binary-word}, the pattern of 0's becomes
\begin{equation}\label{mutated-zero-pattern}
0^{b_1-1}-\cdots -0^{b_i}- \cdots - 0^{b_s}\textbf{0}-.
\end{equation}
Comparing \eqref{mutated-zero-pattern} with \eqref{initial-zero-pattern}, one sees that the leftmost 0 has been bumped out, while we add a new 0 from the right. We remember this new 0 as the $(m+1)$-th, and show it in bold. Note that the $(m+1)$-th 0 is followed by a $-$. This is not an accident: if one continues the analysis for one more step, one sees that the $(m+2)$-th 0 will be separated from the $(m+1)$-th.
On the other hand, the pattern of 1's becomes
\begin{equation}\label{mutated-one-pattern}
-\textbf{1}1^{a_1}- \cdots -1^{a_i}- \cdots -1^{a_s-1}.
\end{equation}
Comparing \eqref{mutated-one-pattern} with \eqref{initial-one-pattern}, one sees that the rightmost 1 has been bumped out, while we add a new 1 from the left. We remember this new 1 as the $(n+1)$-th, and show it in bold. Note that the $(n+1)$-th 1 is preceded by a $-$. Again, this is not an accident: if one continues the analysis for one more step, one sees that the $(n+2)$-th 1 will be separated from the $(n+1)$-th 1. Finally, note that there is a unique way to combine \eqref{mutated-zero-pattern} and \eqref{mutated-one-pattern}  in the zig-zag fashion, after which we recover $\Theta(\Psi(I))$.

Recall that in the above analysis, we have assumed that $a_1\geq 1$,
$b_s\geq 1$, $b_1\geq 2$ and $a_s\geq 2$. The important thing here
is that the assumption ``$a_1\geq 1$ and $b_s\geq 1$" guarantees
that $\Theta(I)$ starts with 1 and ends with 0; while the assumption
``$b_1\geq 2$ and $a_s\geq 2$" guarantees that
$\Theta(\Psi(I))$ starts with 0 and ends with 1. Of course,
there are situations where these  assumptions do not hold. However,
in any case, guided by Lemma \ref{lemma-Pan-binary-word}, the
\textbf{Join-Separate rule} always applies:
\begin{itemize}
\item[$\bullet$]
 if $\Theta(I)$ ends with 1, separate the $(m+1)$-th 0 from the $m$-th 0 by inserting a ``$-$" between them;
otherwise, join the $(m+1)$-th 0  from \emph{right} with the $m$-th
0 of $\Theta(I)$.
\item[$\bullet$] if $\Theta(I)$ starts with 0, separate the $(n+1)$-th 1 from the $n$-th 1 by inserting a ``$-$" between them; otherwise, join the $(n+1)$-th 1 from \emph{left} with the $n$-th 1 of $\Theta(I)$.
\end{itemize}
This rule will be the basic component in the proof of Lemma \ref{lemma-key}.

Given an ideal $I$ of $[m]\times [n]$, recall that $\Gamma(I):=\max(I)$ is the antichain corresponding to $I$. The following lemma is immediate.

\begin{lemma}\label{lemma-antichain-size-binary-word}
Let $I$ be an ideal of $[m]\times [n]$. Then $|\Gamma(I)|$ equals the times that ``10" occurs in $\Theta(I)$.
\end{lemma}

%
%\begin{lemma}\label{lemma-ideal-size-binary-word}
%Let $I$ be the lower ideal of $[m]\times [n]$ given by \eqref{binary-word}. Then $|I|=a_1 b_1 +(a_1 +a_2) b_2+ \cdots+(a_1+\cdots +a_s) b_s$.
%\end{lemma}
%
%\begin{proof}
%This follows from counting the points under the Dyck path corresponding to $I$.
%\end{proof}

Suppose that
\begin{equation}\label{Theta-I}
\Theta(I)= 0^{a_1} 1^{b_k-b_{k-1}}0^{a_2-a_1}1^{b_{k-1}-b_{k-2}}\cdots 0^{a_k-a_{k-1}} 1^{b_1},
\end{equation}
where $a_k=m$, $b_k=n$, $a_1\geq 1$ and $b_1\geq 1$. Now let us define two functions for $I$ on the interval $[1, m+n]$. They will be crucial in calculating the sizes of antichains in the $\Psi$-orbit $\mathcal{O}(I)$.
Let $A=\{a_1+1, \cdots, a_k+1\}$, $B=\{m+1+b_1, \cdots, m+1+b_{k-1}\}$. Put
\begin{equation}\label{P-I}
P_I(i) :=
\begin{cases}
0 &
\mbox{ if } i\in A\cup ([m+2, m+n]\setminus B), \\ 1&
\mbox{ if } i\in  ([1, m+1]\setminus A )\cup B.
\end{cases}
\end{equation}
Let $C=\{b_1+1, \cdots, b_k+1\}$, $D=\{n+1+a_1, \cdots, n+1+a_{k-1}\}$. Put
\begin{equation}\label{Q-I}
Q_I(i) :=
\begin{cases}
0 &
\mbox{ if } i\in ([1, n+1]\setminus C)\cup D, \\ -1&
\mbox{ if } i\in C\cup  ([n+2, n+m]\setminus D ).
\end{cases}
\end{equation}

The following lemma will play a key role in forthcoming discussion.
We postpone its proof to the next section so that one can quickly grasp the sketch.

\begin{lemma}\label{lemma-key}
Suppose that $I$ is the lower ideal of $[m]\times [n]$ given by \eqref{Theta-I}. Then for $i\in [1, m+n]$, we have
\begin{equation}\label{antichain-size-P-Q}
\big|\Gamma\big(\Psi^i(I)\big)\big|=k-1+\sum_{j=1}^{i}(P_I(j)+Q_I(j)).
\end{equation}
\end{lemma}

Let us give an eample to illustrate the formula \eqref{antichain-size-P-Q}.

\begin{example}\label{example-P-Q}
Fix $m=3$, $n=7$. Take $a_1=2, a_2=3$, $b_1=4$, $b_2=7$. That is, $I$ is the lower ideal of
$[3]\times [7]$ with binary word $0011101111$. Then one finds that
$$
P_I(i) =
\begin{cases}
1 &
\mbox{ if } i=1, 2, 8; \\ 0&
\mbox{ otherwise;}
\end{cases}
$$
and that
$$
Q_I(i) =
\begin{cases}
-1 &
\mbox{ if } i=5, 8, 9; \\ 0&
\mbox{ otherwise.}
\end{cases}
$$
On the other hand, by Lemma \ref{lemma-Pan-binary-word}, one calculates the binary words for
$\Psi^{i}(I)$ for $i\in [1, 10]$ as follows:
\begin{center}
\begin{tabular}{c|c}
$i$ &   $\Theta\big(\Psi^{i}(I)\big)$ \\ \hline
 $1$ & $0101110111$ \\
 $2$ & $1010111011$ \\
 $3$ & $1101011101$ \\
 $4$ & $1110101110$ \\
 $5$ & $1111010011$ \\
 $6$ & $1111100101$ \\
 $7$ & $0111111010$ \\
 $8$ & $1011111100$ \\
 $9$ & $1100011111$ \\
 $10$ & $0011101111$ \\
\end{tabular}
\end{center}
Then by Lemma \ref{lemma-antichain-size-binary-word} it is direct to check that the formula \eqref{antichain-size-P-Q} holds for each $i\in [1, 10]$.  \hfill\qed
\end{example}
%\begin{rmk}\label{rmk-matrix}
%If one looks the table above as a matrix, then each column contains
%three 0's. We try to conjecture that this holds in general: take any
%$I\in J([m]\times [n])$, if we put the binary word of
%$\texttt{Pan}^{i}(I)$, where $i$ runs from $1$ to $m+n$, as the
%$i$-th row of a square matrix with size $m+n$, then every column of
%it would contain exactly $m$ 0's. If this is indeed the case,
%dividing each entry by $n$, one would get a doubly stochastic
%matrix.
%\end{rmk}

\begin{prop}\label{prop-antichain-A}
Suppose that $I$ is the ideal whose binary word $\Theta(I)$ is given by \eqref{Theta-I}. Then
$$\sum_{i=1}^{m+n} \big|\Gamma\big(\Psi^i(I)\big)\big|=mn.$$
\end{prop}

\begin{proof}
For convenience, we temporarily put $N=m+n$. We have that
\begin{align*}
\sum_{i=1}^{N} \big|\Gamma\big(\Psi^i(I)\big)\big| &=(k-1)N+ \sum_{i=1}^{N}\sum_{j=1}^{i}(P_I(j)+Q_I(j))\\
&= (k-1)N+ \sum_{i=1}^{N}(N+1-i)(P_I(i)+Q_I(i))  \\
&=(k-1)N+ \sum_{i\in  ([1, m+1]\setminus A )\cup B}(N+1-i)-  \sum_{i\in C\cup  ([n+2, N]\setminus D )}(N+1-i) \\
&=(k-1)N+ \sum_{i\in C\cup  ([n+2, N]\setminus D) }i - \sum_{i\in  ([1, m+1]\setminus A )\cup B}i,
\end{align*}
where the first step uses \eqref{antichain-size-P-Q}, the third step cites \eqref{P-I} and \eqref{Q-I}, while the last step uses the fact that both $([1, m+1]\setminus A )\cup B$ and $C\cup  ([n+2, N]\setminus D)$ have cardinality $m$. Now the desired result follows since elementary calculations lead to
$$
\sum_{i\in C\cup  ([n+2, N]\setminus D )}i - \sum_{i\in  ([1, m+1]\setminus A)\cup B}i=mn-(k-1)N.
$$

\end{proof}

Now we are ready to prove Propp and Roby's theorem.

\medskip
\noindent \emph{Proof of Theorem A.}
As we shall see in Remark \ref{rmk-lemma-binary-word-pattern}, the reverse operator $\Psi_{[m]\times [n]}$ has order $m+n$. Therefore, for any $I\in J([m]\times [n])$, it suffices to verify that
\begin{equation}\label{A-target}
\sum_{i=1}^{m+n} \big|\Gamma\big(\Psi^i(I)\big)\big|=mn.
\end{equation}
When $\Theta(I)$ starts with 0 and ends with 1, say of the form \eqref{Theta-I}, this has been done in Proposition \ref{prop-antichain-A}. Our argument works for the other three cases as well.
\hfill\qed
\medskip

\section{Proof of Lemma \ref{lemma-key}}

This section is devoted to proving Lemma \ref{lemma-key}. We adopt the notations in Section 3.
We always suppose that \eqref{Theta-I} holds. Namely, let $I$ be the lower ideal of  $[m]\times [n]$ corresponding to the binary word
\begin{equation}\label{Theta-I-new}
0^{a_1} 1^{b_k-b_{k-1}}0^{a_2-a_1}1^{b_{k-1}-b_{k-2}}\cdots 0^{a_k-a_{k-1}} 1^{b_1},
\end{equation}
where $a_k=m$, $b_k=n$, $a_1\geq 1$ and $b_1\geq 1$.

For any $a\in\mathbb{P}$, we put $$\underbrace{0-0}_{a}:=\underbrace{0-\cdots-0-\cdots-0}_{a}.$$
Note that $\underbrace{0-0}_{1}=0$. The notation $\underbrace{1-1}_{a}$ is defined similarly.
We associate the \textbf{long $0-$ sequence} to $I$ as follows:
\begin{equation}\label{long-0-sequence}
0^{a_1}-0^{a_2-a_1}-\cdots-0^{a_k-a_{k-1}}-\underbrace{0-0}_{b_1}
\underbrace{0-0}_{b_2-b_1}\cdots \underbrace{0-0}_{b_k-b_{k-1}}0^{a_1}-0^{a_2-a_1}-\cdots-0^{a_k-a_{k-1}}-
\end{equation}
Note that this sequence contains $2m+n$ 0's in total, and ends with $-$.
Fix any $i\in [1, m+n]$. Cut out the consecutive segment of \eqref{long-0-sequence} starting with the $(i+1)$-th 0 and ending with the $(i+m)$-th 0, and include the ``$-$" left (resp. right) to  the $(i+1)$-th 0 (resp. $(i+m)$-th 0) if there is such a ``$-$". We call this segment the \textbf{$i$-th $0-$ sequence} for $I$.

In a similar fashion, we associate the \textbf{long $1-$ sequence} to $I$ as follows:
\begin{equation}\label{long-1-sequence}
-1^{b_k-b_{k-1}}-\cdots-1^{b_2-b_1}-1^{b_1}\underbrace{1-1}_{a_k-a_{k-1}}
\cdots \underbrace{1-1}_{a_2-a_1}\underbrace{1-1}_{a_1}-1^{b_k-b_{k-1}}-\cdots-1^{b_2-b_1}-1^{b_1}
\end{equation}
Note that this sequence contains $m+2n$ 1's in total, and \emph{ends} with $-$. Here we always read the long 1$-$ sequence and its consecutive segments \emph{from right to left}. For instance, cut out the consecutive segment of \eqref{long-1-sequence} starting with the first $1$ and the $(n+a_1)$-th 1, we get
$$
\underbrace{1-1}_{a_1}-1^{b_k-b_{k-1}}-\cdots-1^{b_2-b_1}-1^{b_1}.
$$
Fix any $i\in [1, m+n]$. Cut out the consecutive segment of \eqref{long-1-sequence} starting with the $(i+1)$-th 1 and ending with the $(i+n)$-th 1, and include the ``$-$" left (resp. right) to  the $(i+n)$-th 1 (resp. $(i+1)$-th 1) if there is such a ``$-$". We call this segment the \textbf{$i$-th $1-$ sequence} for $I$.

The following lemma reads $\Theta(\Psi^i(I))$ from \eqref{long-0-sequence} and \eqref{long-1-sequence} for any $i\in [1, m+n]$.

\begin{lemma}\label{lemma-binary-word-pattern}
Let $I$ be the lower ideal of $[m]\times [n]$ given by \eqref{Theta-I-new}. Fix any $i\in [1, m+n]$.
There is a unique way to combine the $i$-th $0-$ sequence and the $i$-th $1-$ sequence for $I$ in the zig-zag fashion, and one gets $\Theta(\Psi^i(I))$.

\end{lemma}

\begin{proof}
It amounts to check that \eqref{long-0-sequence} and \eqref{long-1-sequence} meet the requirements from the scratch. Initially, the first $m$ 0's in \eqref{long-0-sequence} are obtained by replacing each consecutive part of 1's in \eqref{Theta-I-new} with a $-$. Similarly, the first $n$ 1's in \eqref{long-1-sequence} (counted from right to left) are obtained by replacing each consecutive part of 0's in \eqref{Theta-I-new} with a $-$. Recall that $a_k=m$ and $b_k=n$.

Suppose that  $\Theta(\Psi^{i-1}(I))$ has been settled. Let
us consider $\Theta(\Psi^i(I))$. Note that
$\Theta(\Psi^{i-1}(I))$ ends with 0 if and only if the
$(i-1)$-th 1$-$ sequence starts with 1$-$ (recall that we read this
sequence from right to left); $\Theta(\Psi^{i-1}(I))$ starts
with 1 if and only if the $(i-1)$-th 0$-$ sequence starts with $-$0.
Thus the Join-Separate rule says that
\begin{itemize}
\item[$\bullet$]
 if the $(i-1)$-th 1$-$ sequence starts with ``1", add ``$-$0" to  the right side of the $(m+i-1)$-th 0;
 if the $(i-1)$-th 1$-$ sequence starts with ``1$-$", add ``0" to  the right side of the $(m+i-1)$-th 0.
\item[$\bullet$] if the $(i-1)$-th 0$-$ sequence starts with ``0", add ``1$-$" to the left side of the $(n+i-1)$-th 1; if the $(i-1)$-th 0$-$ sequence starts with ``$-$0", add ``1" to the left side of the $(n+i-1)$-th 1.
\end{itemize}

\textbf{(A).}  Firstly, let us analyze the pattern of the next $n$ 0's. Since the first $n$ 1's (counted from right to left) are of the form $$-1^{b_k-b_{k-1}}-\cdots-1^{b_2-b_1}-1^{b_1},$$
the $(m+1)$-th 0 to the $(m+n)$-th 0 are of the form
\begin{equation}\label{0-middle}
-\underbrace{0-0}_{b_1}
\underbrace{0-0}_{b_2-b_1}\cdots \underbrace{0-0}_{b_k-b_{k-1}}
\end{equation}
by the Join-Separate rule. This agrees with
\eqref{long-0-sequence} up to the $(m+b_k)$-th 0.

\textbf{(B).}  Secondly, let us switch to analyze the pattern of the next $m$ 1's.  Since the first $m$ 0's are of the form
$$
0^{a_1}-0^{a_2-a_1}-\cdots - 0^{a_k-a_{k-1}}-,
$$
the $(n+1)$-th 1 to the $(n+m)$-th 1 (counted from right to left) are of the form
\begin{equation}\label{1-middle}
\underbrace{1-1}_{a_k-a_{k-1}}
\cdots \underbrace{1-1}_{a_2-a_1}\underbrace{1-1}_{a_1}-
\end{equation}
by the Join-Separate rule. This agrees with
\eqref{long-1-sequence} up to the $(n+a_k)$-th 1.

\textbf{(C).}  Thirdly, let us come back to analyze the pattern of the further next $m$ 0's. Since the $(n+1)$-th 1 to the $(n+m)$-th 1 are given by \eqref{1-middle}, the $(m+n+1)$-th 0 to the $(m+n+a_k)$-th 0 are of the form
$$0^{a_1}-0^{a_2-a_1}-\cdots-0^{a_k-a_{k-1}}$$
by the Join-Separate rule. This agrees with
\eqref{long-0-sequence} up to the $(m+b_k+a_k)$-th 0.

\textbf{(D).}  Fourthly,  let us switch to analyze the pattern of the further next $n$ 1's.
Since the $(m+1)$-th 0 to the $(m+n)$-th 0 are of the form \eqref{0-middle}, the $(n+m+1)$-th 1 to the $(n+m+n)$-th 1 (counted from right to left) are of the form
$$
1^{b_k-b_{k-1}}-\cdots-1^{b_2-b_1}-1^{b_1}
$$
by the Join-Separate rule. This agrees with
\eqref{long-1-sequence} up to the $(n+a_k+b_k)$-th 1.

\textbf{(E).}  Finally, the $(m+n)$-th 1$-$ sequence starts with
$1$. Thus we should add  ``$-$0" to the $(2m+n)$-th 0 according to
the Join-Separate rule. This agrees with the last $-$ of
\eqref{long-0-sequence}. Similarly, the $(m+n)$-th 0$-$ sequence
starts with 0. Thus we should add ``1$-$" to the left side of the
$(m+2n)$-th 1 in view of the Join-Separate rule. This agrees
with the last $-$ of \eqref{long-1-sequence}, and the proof
finishes.
\end{proof}
\begin{rmk}\label{rmk-lemma-binary-word-pattern}
Note that the five steps above are carried out in the zig-zag way. Moreover, Lemma \ref{lemma-binary-word-pattern} also says that $\Psi^{m+n}(I)=I$. Then it follows immediately that the reverse operator $\Psi_{[m]\times [n]}$ has order $m+n$. This partially recovers Theorem 2 of Fon-Der-Flaass \cite{F} saying that $\mathcal{O}(I)$ has length $(m+n)/d$ for some $d$ dividing both $m$ and $n$.
\end{rmk}

\begin{example}
Let us revisit Example \ref{example-P-Q}, where $a_1=2, a_2=3$, $b_1=4$, $b_2=7$ and $I$ is the lower ideal of
$[3]\times [7]$ with binary word $0011101111$. Now the long $0-$sequence and $1-$sequence for $I$ are
$$
00-0-0-0-0-00-0-000-0-
$$
and
$$
-111-111111-1-111-1111,
$$
respectively. The third $0-$sequence and $1-$ sequence are
$$
-0-0-0-,\quad  11-1-111-1.
$$
Combining them in the zig-zag fashion gives $1101011101$, which agrees with the third row of the table in Example \ref{example-P-Q}. Similarly, the tenth $0-$sequence and $1-$ sequence are
$$
00-0-,\quad  -111-1111.
$$
Combining them in the zig-zag fashion gives $0011101111$, concurring with the last row of the table in Example \ref{example-P-Q}.
 \hfill\qed
\end{example}

\begin{lemma}\label{lemma-antichain-one-step-P-Q}
Let $I$ be the lower ideal of $[m]\times [n]$  given by \eqref{Theta-I-new}. For any $i\in [1, m+n]$, we have
\begin{equation}\label{antichain-one-step-P-Q}
\big|\Gamma\big(\Psi^i(I)\big)\big|=
\big|\Gamma\big(\Psi^{i-1}(I)\big)\big|+P_I(i)+Q_I(i).
\end{equation}
\end{lemma}

\begin{proof}
In view of Lemmas \ref{lemma-antichain-size-binary-word} and  \ref{lemma-binary-word-pattern}, we should analyze the difference between the number of ``10"s
in $\Theta(\Psi^i(I))$ and that in $\Theta(\Psi^{i-1}(I))$. Going from $\Theta(\Psi^{i-1}(I))$ to $\Theta(\Psi^i(I))$, we shall \emph{delete} the $i$-th 1 in \eqref{long-1-sequence} and \emph{add} the $(n+i)$-th 1 in \eqref{long-1-sequence} to form the $i$-th 1$-$ sequence for $I$.

Deleting the $i$-th 1  decreases the number of ``10"s by one if and only if the $i$-th 1 has the form 1$-$ in \eqref{long-1-sequence}. This is measured precisely by the value $Q_I(i)$ defined in \eqref{Q-I}. Similarly, adding the $(n+i)$-th 1  increases the number of ``10"s by one if and only if the $(n+i)$-th 1 has the form 1$-$ in \eqref{long-1-sequence}. This is measured precisely by the value $P_I(i)$ defined in \eqref{P-I}. Now \eqref{antichain-one-step-P-Q} follows.
\end{proof}
\begin{rmk}
The proof above tells us that $P_I(i)=-Q_I(i+n)$.
\end{rmk}

Now \eqref{antichain-size-P-Q} follows directly from \eqref{antichain-one-step-P-Q} and Lemma \ref{lemma-key} is established.

At the end of this section, let us briefly compare our approach to Theorem A with that of Propp and Roby \cite{PR}. On one hand, Propp and Roby cleverly associate a Stanley-Thomas word to each antichain of $[m]\times [n]$ and then $\Psi$ becomes $C_R$---the rightward cyclic shift. The equivariance of the bijection is proved in Proposition 26 of \cite{PR}. Then Theorem A follows quickly.

 On the other hand, we naively associate a binary word to each antichain of $[m]\times [n]$ accroding to the direction (which is either NE or SW) of the path cutting out the corresponding ideal. Then $\Psi$ becomes $\psi$ and the equvariance quickly follows (see Lemma \ref{lemma-Pan-binary-word}). However, compared with $C_R$, the map $\psi$ lacks an easy interpretation. Instead, we separate the zeros and ones in the binary word, and use the Join-Separate rule (which comes from the map $\psi$) to move them forward. In this way, we obtain the long $0-$sequence and the long $1-$sequence. This process is analyzed in Lemma \ref{lemma-binary-word-pattern}. Then we do some counting using the functions $P_I(i)$ and $Q_I(i)$, and arrive at Theorem A.

\section{Type I orbits in $[m]\times K_{n-1}$}
This section aims to consider type I orbits in $[m]\times K_{n-1}$. Note that any full rank ideal of $[m]\times K_{n-1}$ must have the form
\begin{equation}
(L_{2n-1}^{n_0}, L_{i_1}^{n_1}, \dots, L_{i_s}^{n_s})
\end{equation}
where $0\leq i_s<\cdots <i_1<2n-1$, $n_0\in \bbN$, $n_i\in\mathbb{P}$ ($1\leq i\leq s$) are such that $\sum_{i=0}^{s} n_i =m$ (see Lemma \ref{lemma-operator-ideals-CmP}).
By using NE-SW paths, one sees that these ideals are in bijection with $B(m, 2n-1)$---all the binary words with $m$ 0's and $2n-1$ 1's. We still denote this bijection by $\Theta$.
Recall from \eqref{psi} that there is a map $\psi: B(m, 2n-1)\to B(m, 2n-1)$. Similar to Lemma \ref{lemma-Pan-binary-word}, we have the following.

\begin{lemma}\label{lemma-Pan-binary-word-D-I}
Let $I$ be a full rank ideal of $[m]\times K_{n-1}$. Then $\Theta(\Psi(I))=\psi(\Theta(I))$.
\end{lemma}

\begin{lemma}\label{lemma-antichain-size-binary-word-D-I}
Let $I$ be a full rank ideal of $[m]\times K_{n-1}$. Then $|\Gamma(I)|$ equals to the times that ``10" occurs in $\Theta(I)$ plus $\epsilon_n(\Theta(I))$, where $\epsilon_n(\Theta(I))$ equals one if the $n$-th 1 in $\Theta(I)$ is followed immediately by 0, and it equals zero otherwise.
\end{lemma}
\begin{proof}
Note that $\epsilon_n(\Theta(I))=1$ if and only if $L_n$ occurs at some place of $I$, and that the antichain of $K_{n-1}$ corresponding to $L_n=\{1,\cdots, n, n^\prime\}$ contains two elements.
\end{proof}

\begin{prop}\label{prop-D-I}
The average size of antichains in any type I orbit
of $[m]\times K_{n-1}$ equals $\frac{2mn}{m+2n-1}$.
\end{prop}

\begin{proof}
In view of Lemmas \ref{lemma-Pan-binary-word-D-I} and \ref{lemma-antichain-size-binary-word-D-I}, we can work on $[m]\times [2n-1]$, with the exception that whenever the $n$-th 1 is followed immediately by 0, we should add one to the antichain size.  Therefore, let us focus on the long-1-sequence \eqref{long-1-sequence} with $a_k=m$, while $b_k=2n-1$ instead. Recall that we always count the long-1-sequence from right to left. In view of Proposition \ref{prop-antichain-A}, it remains to prove that there are exactly $m$ 1's which are preceded by ``$-$" in a period of \eqref{long-1-sequence}.  It is easy to observe that the indices of these 1's are as follows:
\begin{align*}
&b_1+1,\dots, b_k+1, [b_k+2,b_k+a_1], [b_k+a_1+2, b_k+a_2], \dots, [b_k+a_{k-1}+2, b_k+a_k],\\
&b_k+a_k+b_1+1,\dots,  b_k+a_k+b_k+1, \dots
\end{align*}
To focus on a period, let us count those living on the interval $[n, m+3n-2]$. The total number is
\begin{align*}
&\#\{j\in[k]\mid b_j+1\geq n\} +(a_k-k)+\#\{j\in[k]\mid b_k+a_k+b_j+1\leq m+3n-2\}\\
&=(m-k)+\#\{j\in[k]\mid b_j\geq n-1\}+\#\{j\in[k]\mid b_j\leq n-2\}\\
&=(m-k)+k=m,
\end{align*}
as desired.
\end{proof}

\section{Type II orbits in $[m]\times K_{n-1}$}
This section aims to consider type II orbits in $[m]\times K_{n-1}$.
Recall that any non-full-rank ideal of $[m]\times K_{n-1}$ either has the form \eqref{non-full-rank-1} or the form \eqref{non-full-rank-2}, and that we have introduced the duality ``$\sim$" between these ideals and type II orbits at the end of Section 2. It is obvious that two dual non-full-rank ideals have the same size, and their anticains have the same size as well. Thus it does no harm to study non-full-rank ideals and their orbits up to duality.

We let $L_*$ to be $I_n$ or $I_{n^\prime}$, and view the subscript * as a number between $n-1$ and $n$. Interpreted in this way,  one sees that $L_*$ behaves like a rank ideal by comparing Lemmas \ref{lemma-operator-ideals-CmP} and \ref{lemma-operator-ideals-CmK}. More precisely, the two cases in Lemma \ref{lemma-operator-ideals-CmK} can be written uniformly as
$$
\cdots L_{\max\{*,\, j_1+1\}}^{n_s} L_{\min\{*,\, j_1+1\}}^{m_0}  \cdots.
$$
Here we only present the middle part. Thinking in this way allows us to use a NE-SW path to represent a non-full-rank ideal of $[m]\times K_{n-1}$ up to duality: Instead of $[m]\times [2n-1]$, now we use $[m]\times [2n]$, with the rank * inserted between rank $n-1$ and rank $n$. See  Fig.~4 for an example.

\begin{figure}
\centering \scalebox{0.6}{\includegraphics{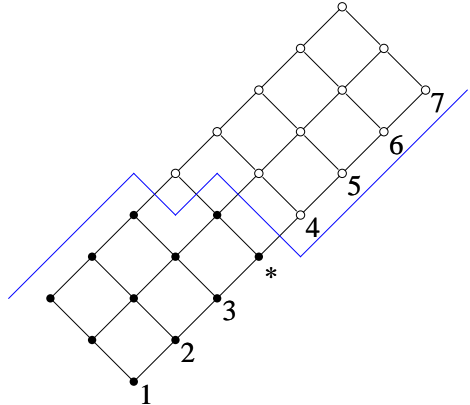}}
\caption{The NE-SW path for the ideal $(I_4^2, L_3)$ in $[3]\times K_3$}
\end{figure}

Let $\overline{B}(m, 2n)$ be the set of binary words with $m$ 0's  and $2n$ 1's, and that the $n$-th 1 is followed immediately by 0. Since the integer $m_0$ in \eqref{non-full-rank-1} or  \eqref{non-full-rank-2} is always positive,  the process gives a bijection
\begin{equation}\label{Theta-D-II}
\Theta: \{\mbox{non-full-rank ideals of } [m]\times K_{n-1}\}/\sim \quad \longrightarrow \quad \overline{B}(m, 2n).
\end{equation}
For instance, the equivalence class of \eqref{non-full-rank-1} is mapped to the following binary word:
$$
1^{j_t}0^{m_t}1^{j_{t-1}-j_t}0^{m_{t-1}}\cdots 1^{j_1-j_2}0^{m_1}1^{n-j_1}0^{m_0}1^{i_s+1-n}0^{n_s}\cdots 1^{i_1-i_2}0^{n_1}1^{2n-1}0^{n_0}.
$$

A general binary word $w$ in $\overline{B}(m, 2n)$ has the form
\begin{equation}\label{binary-word-D-II}
w=1^{a_1}0^{b_1}\dots
 1^{a_i} 0^{b_i}\dots 1^{a_s} 0^{b_s},
\end{equation}
where $s\geq 2$, each $a_i$, $b_i$ is positive except for that  $a_1$ and $b_s$ may be $0$.  By the definition of $\overline{B}(m, 2n)$, there must exist an index $i$ such that $a_1+\cdots+a_i=n$.
We define
\begin{equation}\label{psi-D-II}
\overline{\psi}: \overline{B}(m, 2n) \to \overline{B}(m, 2n)
\end{equation} via
$$
\overline{\psi}(w)=\begin{cases}
0^{b_1-1} 1^{a_1} 0^{b_2+1} 1^{a_2} & \mbox { if } i=1 \mbox{ and }s=2; \\
0^{b_1-1} 1^{a_1} 0^{b_2} 1^{a_2+1} \cdots 0^{b_s+1} 1^{a_s-1}& \mbox { if } i=1 \mbox{ and }s>2; \\
0^{b_1-1}1^{a_1+1}\cdots
 0^{b_i} 1^{a_i}\cdots 0^{b_s+1}1^{a_s-1} & \mbox { if } 1<i \mbox{ and } a_i=1;\\
0^{b_1-1}1^{a_1+1}\cdots
 0^{b_i} 1^{a_i-1}0^{b_{i+1}} 1^{a_{i+1}+1}\cdots 0^{b_s+1}1^{a_s-1} & \mbox { if } 1<i<s-1 \mbox{ and } a_i>1;\\
0^{b_1-1}1^{a_1+1}\cdots
 0^{b_i} 1^{a_i-1}0^{b_s+1}1^{a_s} & \mbox { if } 1<i=s-1 \mbox{ and } a_i>1.
\end{cases}
$$

The following result is a translation of Lemma \ref{lemma-operator-ideals-CmK} into the language of binary words. This time we omit the details.

\begin{lemma}\label{lemma-Pan-binary-word-D-II}
Let $I$ be a non-full-rank ideal of $[m]\times K_{n-1}$. Then
$$
\Theta(\Psi(I))=\overline{\psi}(\Theta(I)).
$$
\end{lemma}

Note that although there are five cases in the definition of \eqref{psi-D-II}, the pattern of 0's is always the same as that of \eqref{psi}.  Therefore,
guided by Lemma \ref{lemma-Pan-binary-word-D-II}, the
\textbf{Join-Separate rule} always applies:
\begin{itemize}
\item[$\bullet$]
 if $\Theta(I)$ ends with 1, separate the $(m+1)$-th 0 from the $m$-th 0 by inserting a ``$-$" between them;
otherwise, join the $(m+1)$-th 0  from \emph{right} with the $m$-th
0 of $\Theta(I)$.
\end{itemize}
As in Section 4, the long-0-sequence can be formed by the above Join-Separate rule, and it records the zero patterns in $\Theta(\Psi^i(I))$ for $i\geq 0$.

Given any binary word $w$ in $\overline{B}(m, 2n)$, the $n$-th 1 always indicates the occurrence of $L_*$ in the ideal $I$ of $[m]\times K_{n}$ corresponding to $w$. Replacing the $n$-th 1 in $w$ by $*$ establishes a bijection between $\overline{B}(m, 2n)$ and $B_*(m, 2n-1)$. Here $B_*(m, 2n-1)$ consists of binary words with $m$ 0's, $2n-1$ 1's and a ``*"; moreover, we require that there are $n-1$ 1's before *, and * is followed immediately by 0. For instance, the element in $B_*(m, 2n-1)$ corresponding to the binary word \eqref{binary-word-D-II} in $\overline{B}(m, 2n)$ is
\begin{equation}
1^{a_1}0^{b_1}\dots
 1^{a_i-1}* 0^{b_i}\dots 1^{a_s} 0^{b_s},
\end{equation}
where $1^{a_i-1}$ is viewed as the empty word if $a_i=1$. Separating 1's and the * from the above equation, we get  its \textbf{$1-*$ pattern}:
\begin{equation}\label{1-*-pattern}
1^{a_1}-\dots-
 1^{a_i-1}* -\dots -1^{a_s}.
\end{equation}
Here we omit the last possible $-$. Moving from $\Theta(I)$ to $\Theta(\Psi(I))$, it is direct to check via Lemma \ref{lemma-Pan-binary-word-D-II} that the $1-*$ pattern becomes
\begin{equation}\label{p-1}
p(1^{a_1+1}-\cdots-1^{a_i-1}*-\cdots-1^{a_s-1})
\end{equation}
if $a_1\geq 1$, and it becomes
\begin{equation}\label{p-2}
p(1-1^{a_1}-\cdots-1^{a_i-1}*-\cdots-1^{a_s-1})
\end{equation}
if $a_1=0$. Here in both cases, $-1^{a_s-1}$ is interpreted as the empty word if $a_s=1$. Moreover, the operator $p$ firstly exchanges * and the $n$-th 1, then it moves  the 1 immediately following * (if there exists) to the next segment. For instance, when $n=4$, we have
$$
p(1-111*-11-1)=1-11*-111-1, \quad p(111-1-*-111)=111-*-1-111.
$$
Therefore, after adding a new 1, the operator $p$ \emph{pushes} the $n$-th 1 through the *. Note that this process does not change the relative patterns of the 1's preceding *. For instance, after adding $n-1$ 1's to \eqref{1-*-pattern} one by one, in the form $1$ or $1-$, and applying $p$ once after adding a 1, then the $n-1$ 1's before * will be pushed after *. However, their relative pattern is still
$$
1^{a_1}-\dots-
 1^{a_i-1}.
$$
Therefore, we should pay attention to how the new 1's are added to \eqref{1-*-pattern}. By the analysis around \eqref{p-1} and \eqref{p-2}, this process obeys the
\textbf{Join-Separate rule}:
\begin{itemize}
\item[$\bullet$]
 if $\Theta(I)$ starts with 0, add the new 1 from left in the form ``$1-$";
otherwise, add it from left in the form ``$1$".
\end{itemize}
Similar to Section 4, we can form the long-1-sequence by consecutively applying the above rule. However, we emphasize that unlike the $[m]\times [n]$ situation, now the long-1-sequence no longer aims to record the ones pattern in $\Theta(\Psi^i(I))$ for $i\geq 0$. Instead, it simply records how the new ones are added in.

The following lemma is immediate.
\begin{lemma}\label{lemma-antichain-size-binary-word-D-II}
Let $I$ be a non-full-rank ideal of $[m]\times K_{n-1}$. Then $|\Gamma(I)|$ equals to the times that ``10" occurs in $\Theta(I)$.
\end{lemma}

\begin{lemma}\label{lemma-long-zero-sequence-D-II}
Let $I$ be the non-full-rank ideal of  $[m]\times K_{n-1}$ corresponding to the binary word
\begin{equation}\label{D-II-key-lemma}
0^{a_1} 1^{b_1}0^{a_2}1^{b_2}\cdots 0^{a_s} 1^{b_s},
\end{equation}
where $\sum_{i=1}^s a_i=m$, $\sum_{i=1}^s b_i=2n$, $a_1\geq 1$, $b_1\geq 1$ and $\sum_{i=1}^k b_i=n$. Then the long-0-sequence associated to $I$ is
\begin{equation}\label{long-0-sequence-D-II-key-lemma}
0^{a_1}-0^{a_2}-\cdots-0^{a_s}-\underbrace{0-0}_{b_s}
\cdots \underbrace{0-0}_{b_{k+2}}\underbrace{0-0}_{b_{k+1}}(-0)^{b_k-1}\underbrace{0-0}_{b_{k-1}} \cdots  \underbrace{0-0}_{b_1}   0^{a_1}-0^{a_2}-\cdots-0^{a_s}-
\end{equation}
Here $(-0)^{b_k-1}$ is interpreted as the empty word if $b_k=1$.
\end{lemma}
\begin{proof}
Let us start with
\begin{equation}\label{first-m-0-D-II}
0^{a_1}-0^{a_2}-\cdots -0^{a_s}
\end{equation}
the zeros pattern of \eqref{D-II-key-lemma}.

\textbf{(A).} The $1-*$ pattern of \eqref{D-II-key-lemma} is
\begin{equation}\label{1-*-key-lemma}
1^{b_1}-1^{b_2}-\cdots-1^{b_{k-1}}-1^{b_k-1}*-1^{b_{k+1}}-\cdots -1^{b_s}.
\end{equation}
The last $n$ 1's will be bumped out one by one in the form of
$$
-1^{b_{k+1}}-\cdots -1^{b_s}.
$$
Therefore, by the Join-Separate rule, the next $n$ zeros will be added to \eqref{first-m-0-D-II} in the form
$$
-\underbrace{0-0}_{b_s}
\cdots \underbrace{0-0}_{b_{k+2}}\underbrace{0-0}_{b_{k+1}}.
$$

\textbf{(B).} Now let us handle the next $n-1$ step. As noted earlier, the $n-1$ 1's before * in \eqref{1-*-key-lemma} will be pushed through *, and they will be bumped out in the relative form
$$
-1^{b_1}-1^{b_2}-\cdots-1^{b_{k-1}}-1^{b_k-1}.
$$
Therefore, by the Join-Separate rule, the next $n-1$ zeros will be added  in the form
$$
(-0)^{b_k-1}\underbrace{0-0}_{b_{k-1}} \cdots  \underbrace{0-0}_{b_2}\underbrace{0-0}_{b_1}.
$$

\textbf{(C).} Now let us analyze the last $m$ 0's. Since the first $m$ 0's are in the form \eqref{first-m-0-D-II}, the first $m$ 1's will be added to \eqref{1-*-key-lemma} in the form
$$
\underbrace{1-1}_{a_s} \cdots \underbrace{1-1}_{a_2}\underbrace{1-1}_{a_1}-.
$$
Again they will be pushed through *, and then bumped out in the above relative form. Therefore, by the Join-Separate rule, the last $m$ 0's will be added in the form
$$
 0^{a_1}-0^{a_2}-\cdots-0^{a_s}.
$$
This finishes the proof.
\end{proof}
\begin{rmk}\label{rmk-lemma-long-zero-sequence-D-II}
In the setting of Lemma \ref{lemma-long-zero-sequence-D-II}, the binary word $\Theta(\Psi^{m+2n-1}(I))$ has the same zeros pattern as that of $\Theta(I)$, namely, \eqref{first-m-0-D-II}.
On the other hand, the long-1-sequence goes like
$$
1^{b_1}-\cdots-1^{b_k-1}-\cdots -1^{b_s}\underbrace{1-1}_{a_s} \cdots \underbrace{1-1}_{a_2}\underbrace{1-1}_{a_1}-1^{b_1}-\cdots-1^{b_k-1}*-\cdots -1^{b_s}.
$$
Since pushing a sequence of ones through * does not change their relative pattern, one sees that $\Theta(\Psi^{m+2n-1}(I))$ has the same $1-*$ pattern as that of $\Theta(I)$, namely, \eqref{1-*-key-lemma}. Thus we conclude that $\Psi^{m+2n-1}(I)=I$. Then it is immediate that $\Psi_{[m]\times K_{n-1}}$ has order $m+2n-1$. This result should already be known by Rush and Shi. Indeed, in Theorem 10.1 of \cite{RS}, the cyclic sieving phenomenon defined in \cite{RSW} has been established for $[m]\times K_{n-1}$.
\end{rmk}

\begin{prop}\label{prop-D-II}
The average size of antichains in any type II orbit
of $[m]\times K_{n-1}$ equals $\frac{2mn}{m+2n-1}$.
\end{prop}

\begin{proof}
Without loss of generality, let $I$ be the non-full-rank ideal of  $[m]\times K_{n-1}$ corresponding to the binary word \eqref{D-II-key-lemma}. According to Remark \ref{rmk-lemma-long-zero-sequence-D-II},  $\Psi_{[m]\times K_{n-1}}$ has order $m+2n-1$. One can consecutively cut off $m+2n-1$ segments of zeros from \eqref{long-0-sequence-D-II-key-lemma}, each having length $m$. It remains to show that the total number of $-0$'s in these $m+2n-1$ segments is  $2mn$. On the other hand, one can  consecutively cut off $m+2n$ segments of zeros, each having length $m$, from the following sequence:
$$
0^{a_1}-0^{a_2}-\cdots-0^{a_s}-\underbrace{0-0}_{b_s}
\cdots \underbrace{0-0}_{b_{k+1}} \underbrace{0-0}_{b_k} \cdots \underbrace{0-0}_{b_1}   0^{a_1}-0^{a_2}-\cdots-0^{a_s}-
$$
The number we want is the same as the number of $-0$'s in the $m+2n$ segments.
By Proposition \ref{prop-antichain-A}, the latter number is $2mn$. This finishes the proof.
\end{proof}

\section{Certain root posets}

In this section, we shall introduce a bit Lie theory, and illustrate that the several infinite families of posets in Section 1 occur as sub-families of certain root posets.

Let $\frg$ be a finite-dimensional simple Lie algebra over $\bbC$. Fix a Cartan subalgebra $\frh$ of $\frg$.
The associated root system is $\Delta=\Delta(\frg, \frh)\subseteq\frh^*$. Fix a set of positive roots $\Delta$, and let $\Pi=\{\alpha_1, \dots, \alpha_l\}$ be the corresponding simple roots.  Let $[\alpha_i]$ be the set of all positive roots $\alpha$ such that the coefficient of $\alpha_i$ in $\alpha$ is one. Note that $[\alpha_i]$
inherits a poset
structure from the usual one of $\Delta^+$: let $\alpha$
and $\beta$ be two roots of $[\alpha_i]$, then $\alpha\leq\beta$ if
and only if $\beta-\alpha$ is a nonnegative integer combination of
simple roots. We call the posets $[\alpha_i]$ by $\Delta(1)$. The latter notation reflects the fact that they come from $\bbZ$-gradings of $\frg$, see \cite{DW} and references therein for more background.

Previously, by using Ringel's paper \cite{R}, we have analyzed the structure of $\Delta(1)$ in Section 4 of \cite{DW}. More precisely, they include  three infinite families
$[m]\times [n]$, $H_n$, $[m]\times K_{n-1}$;
and the following twenty posets:
\begin{itemize}
\item[$\bullet$] $[2]\times [3] \times [3]$, $[2]\times [3] \times [4]$, $[2]\times [3] \times [5]$;
\item[$\bullet$] $[2]\times H_4$, $[3]\times H_4$, $[4]\times H_4$; $[2]\times H_5$, $[2]\times H_6$;
\item[$\bullet$] $J^2([2]\times [3])$, $[2]\times J^2([2]\times [3])$, $[3]\times J^2([2]\times [3])$; $J^3([2]\times [3])$, $[2]\times J^3([2]\times [3])$;
\item[$\bullet$] $[\alpha_4]$ in $F_4$; $[\alpha_2]$ in $E_6$; $[\alpha_1]$
and $[\alpha_2]$ in $E_7$;  $[\alpha_1]$,
$[\alpha_2]$ and $[\alpha_8]$ in $E_8$.
\end{itemize}
The above posets all come from exceptional Lie algebras.  Here we label the simple roots as Knapp \cite{K}.

Finally, let us raise two conjectures pertaining to the orbit structure of $\Delta(1)$.
Let $w_0^i$ be the longest element of the subgroup $\langle s_{\alpha_j}\mid j\neq i\rangle$ of the Weyl group. Here $s_{\alpha_j}$ is the reflection according to the simple root $\alpha_j$. Then $w_0^i$ permutes the roots in $\Delta(1)=[\alpha_i]$. Given a root $p$ of $\Delta(1)$, we denote $w_0^i(p)$ by $p^*$.
For any $p\in\Delta(1)$, and for any orbit $\mathcal{O}$,
define $M_{\mathcal{O}}(p)$ to be the number of times that $p$
occurs in the ideals of $\mathcal{O}$. That is,
\begin{equation}\label{mult-number}
M_{\mathcal{O}}(p):=\big|\{ I \in\mathcal{O} \mid p\in I\}\big|.
\end{equation}

\begin{conj}\label{conj-ideal}
For any $p\in\Delta(1)$, we
have
\begin{equation}\label{conservation-law}
M_{\mathcal{O}}(p) + M_{\mathcal{O}}(p^*)=\big|\mathcal{O}\big|.
\end{equation}
\end{conj}
\medskip

%Suppose that
%\eqref{conservation-law} holds. Then summing both sides over $p$ in
%$\Delta(1)$ would give
%\begin{align*}
%\big|\Delta(1)\big|\cdot \big|\mathcal{O}\big| &= \sum_{p\in \Delta(1)} (M_{\mathcal{O}}(p)+M_{\mathcal{O}}(p^*))\\
% &= 2\sum_{p\in \Delta(1)} M_{\mathcal{O}}(p)\\
%&= 2\sum_{I\in \mathcal{O}} |I|.
%\end{align*}
%Thus it would follow that
%$$
%\frac{1}{\big|\mathcal{O}\big|}\sum_{I\in \mathcal{O}}
%|I|=\frac{1}{2} \big|\Delta(1)\big|.
%$$

Similarly, we
define $N_{\mathcal{O}}(p)$ to be the number of times that $p$
occurs in the antichains of $\mathcal{O}$. That is,
\begin{equation}\label{antichain-number}
N_{\mathcal{O}}(p):=\big|\{ \Gamma \in\mathcal{O} \mid p\in
\Gamma\}\big|.
\end{equation}

\begin{conj}\label{conj-antichain}
For any $p\in\Delta(1)$, we have
\begin{equation}\label{symm-dist}
N_{\mathcal{O}}(p) = N_{\mathcal{O}}(p^*).
\end{equation}
\end{conj}

\begin{rmk}\label{rmk-homomesy}
We mention that the previous two conjectures can be formulated in the language of homomesy invented by Propp and Roby \cite{PR}.
Indeed, let $\chi_{p}$ be the function on $J(\Delta(1))$ such that $\chi_{p}(I)=1$ if $p\in I$ and $\chi_{p}(I)=0$ otherwise. Then Conjecture \ref{conj-ideal} amounts to the claim that the function $\chi_{p}+\chi_{p^*}$ is $1$-mesic for any $p\in\Delta(1)$ under the reverse operator. Similarly, in the antichain setting, let $\chi^\prime_{p}$ be the function on $\mathrm{An}(\Delta(1))$ such that $\chi^\prime_{p}(A)=1$ if $p\in A$ and $\chi^\prime_{p}(A)=0$ otherwise. Then Conjecture \ref{conj-antichain} is exactly the claim that the function $\chi^\prime_{p}-\chi^\prime_{p^*}$ is 0-mesic for any $p\in \Delta(1)$ under the reverse operator.
Therefore, the two conjectures amount to guess that $\Delta(1)$ could be a good place to exhibit homomesy.
\end{rmk}

\medskip

\centerline{\scshape Acknowledgements}
We are grateful to Tom Roby and David Rush for very helpful discussions.
We thank James Propp  and  Hugh Thomas sincerely for their nice suggestions which allow us to understand homomesy better. Part of the work was carried out during Dong's visit of MIT. He thanks its Math Department sincerely for offering excellent working conditions.

\end{document}